\documentclass[12pt,a4paper,oneside]{amsart}

\usepackage{amsfonts, amsmath, amssymb, amsthm, hyperref}
\usepackage{anysize}

\newtheorem{thm}{Theorem}[section]
\newtheorem*{thm*}{Theorem}
\newtheorem{lem}[thm]{Lemma}

\newtheorem{cor}[thm]{Corollary}

\theoremstyle{definition}
\newtheorem{defn}[thm]{Definition}

\theoremstyle{remark}
\newtheorem{rem}[thm]{Remark}

\newtheorem{ques}[thm]{Question}

\numberwithin{equation}{section}

\DeclareMathOperator{\vol}{vol}

\renewcommand{\epsilon}{\varepsilon}
\renewcommand{\phi}{\varphi}
\renewcommand{\kappa}{\varkappa}

\begin{document}

\title{An analogue of Gromov's waist theorem for coloring the cube}
\author{R.N.~Karasev}
\thanks{Supported by the Dynasty Foundation, by ERC Advanced Research Grant No.~267195 (DISCONV), the President's of Russian Federation grant MK-113.2010.1, the Russian Foundation for Basic Research grants 10-01-00096 and 10-01-00139, the Federal Program ``Scientific and scientific-pedagogical staff of innovative Russia'' 2009--2013, and the Russian government project 11.G34.31.0053.}
\email{r\_n\_karasev@mail.ru}
\address{Roman Karasev, Dept. of Mathematics, Moscow Institute of Physics and Technology, Institutskiy per. 9, Dolgoprudny, Russia 141700}
\address{Roman Karasev, Laboratory of Discrete and Computational Geometry, Yaroslavl' State University, Sovetskaya st. 14, Yaroslavl', Russia 150000}


\begin{abstract}
It is proved that if we partition a $d$-dimensional cube into $n^d$ small cubes and color the small cubes in $m+1$ colors then there exists a monochromatic connected component consisting of at least $f(d, m) n^{d-m}$ small cubes.
\end{abstract}

\maketitle

\section{Introduction}

One possible way to express that $Q^d = [0,1]^d$ has dimension $d$ is to claim that it cannot be colored in $d$ colors with arbitrarily small connected monochromatic components. In~\cite{matpri2008} the following related question was studied: 

\begin{ques}
\label{coloring-ques}
If we color $Q^d$ in $m+1$ colors (to make the problem discrete we color small cubes of the partition of $Q^d$ into $n^d$ small cubes) then what size of a monochromatic connected component can we guarantee?
\end{ques}

For $m=d-1$ the HEX lemma~\cite{gale1979} (or an appropriate discretization of the Sitnikov theorem~\cite{sit1958} on the Alexandrov waist, see also~\cite[Section~6]{kar2010cpt}) shows that there must exist a monochromatic connected component spanning two opposite facets, and therefore consisting of at least $n$ small cubes. In~\cite{matpri2008} the coloring in $2$ colors was studied using isoperimetric inequalities for the grid and a lower bound $n^{d-1} - d^2n^{d-2}$ for the size of a connected monochromatic component was established. It was conjectured in~\cite{matpri2008} that the size of a monochromatic connected component is of order $n^{d-m}$ for $m+1$ colors. 

Alexey Kanel-Belov also posed the same problem in 1990s (private communication) and it circulated among mathematicians in Moscow, see for example~\cite[Problem~15]{kb2010}. Marsel~Matdinov has independently obtained another solution for this problem that appeared in~\cite{matd2011}.

Here we give a nontrivial lower bound for the size of a monochromatic connected component:

\begin{thm}
\label{cube-color}
Suppose that a $d$-dimensional cube $Q^d$ is partitioned into $n^d$ small cubes in an obvious way. Let $0\le m < d$. If the set of small cubes of $Q^d$ is colored in $m+1$ colors then there exists a connected monochromatic component of size at least
$$
f(d, m) n^{d-m}.
$$
Here $f(d,m)$ is a function depending on $d$ and $m$ and \emph{not depending} on $n$.
\end{thm}

\begin{rem}
In this theorem two small cubes are connected if they have a nonempty intersection as closed sets.
\end{rem}

\begin{rem}
Marsel~Matdinov~\cite{matd2011} has noted that the number of colors can be arbitrary, the only thing we have to check is that no point is colored in more that $m+1$ colors. Indeed, the proof in Section~\ref{proof-sec} of this paper uses the multiplicity of coloring and does not use the total number of colors.
\end{rem}

\begin{rem}
From (\ref{asymptotic-f}) we see that for large enough $n$ we can take $\left( (m+1)! \binom{d}{m} 4^m\right)^{-1}$ as the coefficient in this theorem. The coefficient $(m+1)!$ is an inevitable consequence of the applied technique, as it was in~\cite{grom2010,kar2011}, while the part $\binom{d}{m} 4^m$ still may be improved. It would be natural to have an estimate not depending on $d$, as it is for $m=1$ and $m=d-1$.
\end{rem}

\begin{rem}
From Lemma~\ref{bounded-vol} below it will be clear that the cubic partition may not be the optimal one in this problem. In the proof of Theorem~\ref{cube-color} the partition has to be put into a general position thus spoiling the constant $f(d, m)$. A \emph{simple} (which is \emph{dual to simplicial}) partition of $Q^d$ into small parts of bounded complexity may be more suitable. We could also start (following~\cite{matpri2008}) from a partition Poincar\'e dual (which means consisting of stars of vertices in the barycentric subdivision) to a triangulation of $Q^d$ obtained by triangulating every small cube of the standard cubic partition; such a partition is already simple.
\end{rem}

The result of this paper and the method of its proof has very much in common with the results in~\cite{grom2003,mem2009,grom2010} on ``waists''. Let us remind two of them:

\begin{thm}[Gromov--Memarian, 2003--2009]
\label{sph-waist-strong}
Suppose $f : S^d \to \mathbb R^m$ is a continuous map. There exists a point $z\in \mathbb R^m$ such that for any $\varepsilon > 0$
$$
\vol U_\varepsilon(f^{-1}(z)) \ge \vol U_\varepsilon S^{d-m}.
$$
Here $U_\varepsilon$ denotes the $\varepsilon$-neighborhood, $S^d$ is the standard unit sphere, and $S^{d-m}$ is its $(d-m)$-dimensional equatorial subsphere, i.e. $S^{d-m} = S^d\cap \mathbb R^{d-m+1}$.
\end{thm}

\begin{cor}
\label{sph-waist-vol}
In particular, for generic smooth maps $f: S^d\to\mathbb R^d$ it follows that for some point $z\in\mathbb R^d$ the preimage $f^{-1}(z)$ has $(d-m)$-dimensional volume $\vol_{d-m} f^{-1}(z) \ge \vol_{d-m} S^{d-m}$. 
\end{cor}

Since $Q^d$ can be embedded into $S^d$ preserving the volumes up to some constant, Theorem~\ref{cube-color} could follow from Corollary~\ref{sph-waist-vol} directly if there were some lower bounds for maximal volume of a \emph{connected component} of $f^{-1}(z)$. An open problem is to establish such a result.

\medskip
\textbf{Acknowledgments. }
The author thanks Ji\v{r}\'{\i}~Matou\v{s}ek and Imre~B\'{a}r\'{a}ny for intensive discussions and numerous useful remarks.

\section{Filling of cycles in $(Q, \partial Q)$}

Fix the dimension $d$ and denote $Q^d$ by $Q$ for brevity. Generally, we are going to follow the approach of~\cite{grom2010} to examine the \emph{filling profile} for cycles in $(Q,\partial Q)$. Let us work with singular $k$-dimensional chains $C_k(Q)$ of the space $Q$, represented by piece-wise linear (or shortly PL) images of $k$-dimensional simplicial complexes. Assume that $Q$ is the unit cube in the Euclidean space.

\begin{defn}
Denote the $k$-dimensional Euclidean volume of $c\in C_k(Q)$ by 
$$
\|c\| = \sum_{\sigma} |\alpha_\sigma| \vol_k \sigma,
$$
if $c=\sum_\sigma \alpha_\sigma \sigma$ is the linear combination of simplices. 
\end{defn}

There arises the corresponding boundary operator $\partial : C_k(Q)\to C_{k-1}(Q)$, the relative cycle group
$$
Z_k(Q, \partial Q) = \{c\in C_k(Q) : \partial c\in C_{k-1}(\partial Q)\},
$$
the boundary group
$$
B_k(Q) = \partial C_{k+1}(Q) \subset Z_k(Q, \partial Q),
$$
and the homology group $H_k(Q,\partial Q) = Z_k(Q,\partial Q) / B_k(Q)$.

The well known fact is that $H_k(Q,\partial Q) = 0$ for $k\neq d$ and $H_d(Q,\partial Q)=\mathbb Z$.

\begin{rem}
In the rest of the proof we use integral homology and have to be careful about signs. Though everything passes literally with modulo $2$ homology without worrying about signs; the integral homology is just left for other possible applications.
\end{rem}

So the operator $\partial$ is invertible on $Z_k(Q,\partial Q)$ for $k<d$ and the following lemma allows to invert it economically:

\begin{lem}
\label{filling}
For any $z\in Z_k(Q, \partial Q)$ with $0\le k < d$ one can find $H(z)$ such that $\partial H(z) = z\pmod {\partial Q}$ and 
$$
\|H(z)\| \le (k+1) \|z\|.
$$
\end{lem}

\begin{rem}
The value $H(z)$ may not depend on $z$ linearly, but the essential thing is that we have a linear bound on the volume.
\end{rem}

\begin{proof}
Consider all possible intersections $z_t=z\cap \{x_1=t\}$ with a hyperplane orthogonal to the $0x_1$ axis. We have the following inequality:
\begin{equation}
\label{integration}
\int_0^1 \|z_t\| dt \le \|z\|.
\end{equation}
It is sufficient to check it for every simplex $\sigma$ of $z$, when it is easy to check.

From (\ref{integration}) it follows that one of the sections $z_t$ must have $(k-1)$-dimensional volume at most $\|z\|$, and for generic $t$ it must be a $(k-1)$-cycle. Now we cut the cycle $z$ into two chains $z = z_0 + z_1$, where $z_0$ corresponds to $x_1 \le t$ and $z_1$ corresponds to $x_1 \ge t$. The cutting is applicable because after cutting every simplex of $z$ we may triangulate the cut parts to make it again a simplicial chain. It is clear that
$$
\partial z_0 = z_t \pmod{\partial Q}\quad\text{and}\quad  \partial z_1 = -z_t \pmod{\partial Q}.
$$

Let $F_0$ and $F_1$ be the facets of $Q$ with $x_1=0$ and $x_1=1$ respectively. If a chain $y\in C_\ell(Q, \partial Q)$ is given as a PL image of a simplicial complex $y : K\to Q$ then we construct a PL map $I_0(y) : K\times [0, 1] \to Q$ by sending $p\times 1$ to $y(p)$, $p\times 0$ to the projection of $y(p)$ onto $F_0$, and $p\times s$ to the corresponding combination of $I_0(y)(p\times 0)$ and $I_0(y)(p\times 1)$. If $y$ does not touch $F_1$ then we have the following properties:
$$
\partial I_0(y) = y + I_0(\partial y)\pmod{\partial Q}\quad\text{and}\quad \|I_0(y)\| \le \|y\|.
$$
Similarly we define $I_1(y)$ with the projection onto $F_1$ for those $y$ that do not touch $F_0$ with the properties:
$$
\partial I_1(y) = y - I_1(\partial y)\pmod{\partial Q}\quad\text{and}\quad \|I_1(y)\| \le \|y\|.
$$
Now return to our cycle $z=z_0+z_1$ and put
$$
H(z) = I_0(z_1) + I_1(z_2) - H(z_t)\times [0, 1].
$$
The cycle $z_t$ is considered to be in the cube $Q' = Q\cap \{x_1=t\}$ of lower dimension; it has dimension $k-1$ and $H$ is defined for it by the inductive assumption. The multiplication by the segment $[0,1]$ produces $(k+1)$ dimensional chains in $Q$ from chains in $Q'$ in an obvious way.
 
For the boundary we have (modulo $\partial Q$):
$$
\partial H(z) = \partial I_0(z_1) + \partial I_1(z_2) - z_t\times [0,1] = z_0 + I_0(z_t) + z_1 - I_1(z_t) - z_t\times [0,1] = z,
$$
because $I_0(z_t) - I_1(z_t)$ obviously equals $z_t\times [0,1]$. For the volumes we also use the inductive assumption for $H(z_t)$ along with the choice of $t$ to show that:
$$
\|H(z)\| \le \|z_0\| + \|z_1\| + \|H(z_t)\| = \|z\| + \|H(z_t)\| \le \|z\| + k\|z_t\| \le (k+1)\|z\|.
$$
\end{proof}

\begin{defn}
Call a chain $c\in C_k(Q)$ \emph{rectilinear} if all its simplexes are parallel to coordinate $k$-subspaces of $\mathbb R^d$.
\end{defn}

For this type of chains we improve the filling profile:

\begin{lem}
\label{filling-rect}
For any rectilinear $z\in Z_k(Q, \partial Q)$ with $0\le k < d$ one can find a rectilinear $H(z)$ such that $\partial H(z) = z\pmod {\partial Q}$ and 
$$
\|H(z)\| \le \|z\|.
$$
\end{lem}

\begin{proof}
Proceed as in the previous proof. Any simplex of $z$ is either orthogonal to $0x_1$, or is parallel to $0x_1$. Thus the volume splits into corresponding parts:
$$
\|z\| = \|z\|_\perp + \|z\|_\parallel.
$$
When selecting the section $z_t$ we can guarantee that $\|z_t\| \le \|z\|_\parallel$, because in the rectilinear case we have the equality 
$$
\int_0^1 \|z_t\| dt = \|z\|_\parallel.
$$ 
Then
$$
\|H(z)\| \le \|I_0(z_1)\| + \|I_1(z_2)\| + \|H(z_t)\times [0, 1]\| \le \|z\|_\perp + \|z\|_\parallel = \|z\|.
$$
Here we use the fact that $\|I_0(z_1)\| \le \|z_1\|_\perp$ and $\|I_1(z_2)\|\le \|z_2\|_\perp$, along with the inductive assumption.
\end{proof}

\section{Proof of Theorem~\ref{cube-color}}
\label{proof-sec}

Denote by $C_1, \ldots, C_M$ the monochromatic connected components of the coloring. By perturbing slightly the walls of the partition of $Q$ into small cubes we may assume that every intersection $C_{i_0}\cap \dots\cap C_{i_k}$ has dimension at most $d-k$. In this partition of $Q$ into parts some sets $C_i$ may becomes disconnected but this makes the statement even stronger. 

Will will give three possible options for perturbations, because clear understanding of the possibilities may help in finding reasonable bounds for the constant $f(d,m)$ in this theorem:

\begin{itemize}
\item 
The coloring corresponds to a PL map $\chi : Q\to \Delta$, where $\Delta$ is an $m$-dimensional simplex with vertices $w_0,\ldots, w_m$. The subset of $Q$ of color $i$ corresponds to the preimage of the star of $v_i$ in the barycentric subdivision $\Delta'$. If we perturb the map $\chi$ so that it becomes transversal to the triangulation $\Delta'$ then we obtain the required properties of intersections $C_{i_0}\cap \dots\cap C_{i_k}$ because they correspond to preimages of codimension $k$ faces of $\Delta'$. 

\item 
We can also view the partition of $Q$ into smaller cubes as projection of facets of the graph of a convex function $\phi : Q \to \mathbb R$ onto $Q$. This function can be built as follows: take the convex PL function of one variable $\phi_1(x)$ with discontinuity of the derivative exactly at $x=1/n, 2/n,\ldots, (n-1)/n$. Then $\phi(x_1,\ldots, x_d) = \sum_{i=1}^d \phi_1(x_i)$ is exactly what is needed. Then we perturb the graph of $\phi$ to put its facets into a general positions (thus making this graph \emph{simple}) and obtain the PL function $\psi(x)$. The projections of facets of the graph of $\psi$ give a partition of $Q$ arbitrarily close to the cubic partition. The advantages of this method are convexity of parts and reasonable behaviour of faces of the partition in any dimension.

\item 
The third way to make the partition simple is shifting of the partition cubes.\footnote{This method was proposed by Imre~B\'{a}r\'{a}ny and helps to provide reasonable explicit bounds in (\ref{h-d1n} and (\ref{asymptotic-f}).} Suppose in every layer of $n^{d-1}$ cubes we have shifted the cubes so that the corresponding $(d-1)$-dimensional pattern is simple. Then we shift layers relative to each other and the general position will be simple again. Under such a shifting some cubes in the partition will be cut out and we will have to introduce some new (incomplete) partition cubes. This does not affect the proof because we may assume that the volumes of all the faces of the partition changed by arbitrarily small value. Important thing is that all intersections in this partition remain rectilinear.

\item
Another way of making the partition simple \footnote{Also proposed by Imre~B\'{a}r\'{a}ny} is to note that the cubical partition is the Voronoi partition corresponding to the integer grid. Now we can deform the Euclidean norm (keeping the points) so that the corresponding partition becomes simple.
\end{itemize}

Now we consider the nerve $N$ of the covering of $Q$ by $C_i$'s. Since all the simplices of $N$ are heterochromatic using $m+1$ colors in total it follows that $N$ is $m$-dimensional. After the perturbation of the partition and making it simple, for any simplex $(i_0, \ldots, i_m)\in N$ of maximal dimension the corresponding intersection 
$$
C(i_0,\ldots, i_m) = C_{i_0}\cap \dots \cap C_{i_m}
$$
is at most $(d-m)$-dimensional and actually represents a $(d-m)$-dimensional cycle. For a simplex $(i_0, \ldots, i_k)\in N$ of dimension $k < m$ the intersection 
$$
C(i_0,\ldots, i_k) = C_{i_0}\cap \dots \cap C_{i_k}
$$
is $(d-k)$-dimensional, but it may not be a cycle. Its boundary is given by
\begin{equation}
\label{boundary-op}
\partial C(i_0,\ldots, i_k) = \sum_{i_{k+1}} \pm C(i_0, \ldots, i_{k+1}),
\end{equation}
where summation is over $i_{k+1}$ such that $(i_0,\ldots, i_{k+1})\in N$ and the signs are chosen depending on the orientation. In what follows we consider $C(i_0, \ldots, i_k)$ as an antisymmetric expression with respect to permutation of indices, where sign is given by changing the orientation of the chain. In particular, with proper choice of orientations of the chains the sings $\pm$ in (\ref{boundary-op}) will be $+$; these orientations may be considered as pullbacks of certain orientations of the simplices of $\Delta'$ under the coloring map $\chi$. Of course, we may use mod $2$ homology without worrying about signs at all.

Assume that the volume of every $C_i$ is at most $\alpha n^{-m}$ (remember that we normalize by $\|Q\|=1$). We want to bound $\alpha$ from below independently on $n$. We need the lemma:

\begin{lem}
\label{bounded-vol}
For every fixed $i$ the total volume of the $k$-codimensional skeleton
$\|C_i^{(d-k)}\|$ is at most $g(d,k) \|C_i\| n^k \le g(d, k) \alpha n^{k-m}$.
\end{lem}

\begin{proof}
We have started with the cubic sets $C_i$, and for every small cube of volume $n^{-d}$ the total volume of its $k$-codimensional facets is equal to $\binom{d}{k}2^k n^{d-k}$. Every such $k$-codimensional face may split into several faces after perturbing the cubic partition; but this splitting multiplicity can also be estimated. For example, if we use the perturbation by shifting it may split into $2^{k-1}$ parts at most. In this case, by denoting $g(d, k) = \binom{d}{k}2^{2k-1}$ we obtain the required inequality. In other cases an analogous inequality holds with another $g(d,k)$.
\end{proof}

Now we do what is called in~\cite{grom2010} ``contraction in the space of cycles''. We actually use an elementary version of contraction (without using the Dold--Thom--Almgren theorem) similar to the reasoning in~\cite{kar2011}. 

For any $(d-k)$-dimensional chain $C(i_0,\ldots, i_k)$ we define the $(d-k+1)$-dimensional chain $F(i_0,\ldots, i_k)$ so that the following condition holds: 
\begin{equation}
\label{contraction-eq}
\partial F(i_0,\ldots, i_k) = C(i_0,\ldots, i_k) - \sum_{i_{k+1}} F(i_0, \ldots, i_{k+1}).
\end{equation}

In what follows we assume that if $(i_0,\ldots, i_k)$ is not a simplex of $N$ then $C(i_0,\ldots, i_k) = F(i_0,\ldots, i_k) = 0$. We will also build $F(i_0, \ldots, i_k)$ so that it is antisymmetric. Start with
$$
F(i_0, \ldots, i_m) = H\left(C(i_0, \ldots, i_m)\right),
$$
where $H$ is from Lemma~\ref{filling}. Condition (\ref{contraction-eq}) holds obviously. Then proceed by descending induction on $k$. For any boundary formula 
$$
\partial C(i_0,\ldots, i_k) = \sum_{i_{k+1}} C(i_0, \ldots, i_{k+1})
$$
with $0 < k < d$ put
$$
F(i_0,\ldots, i_k) = H \left( C(i_0,\ldots, i_k) - \sum_{i_{k+1}} F(i_0, \ldots, i_{k+1}) \right).
$$
The operator $H$ from Lemma~\ref{filling} can be applied because
\begin{multline*}
\partial \left( C(i_0,\ldots, i_k) - \sum_{i_{k+1}} F(i_0, \ldots, i_{k+1}) \right) =\\= \sum_{i_{k+1}} C(i_0, \ldots, i_{k+1}) - \sum_{i_{k+1}} C(i_0, \ldots, i_{k+1}) + \sum_{i_{k+1}, i_{k+2}} F(i_0,\ldots, i_{k+1}, i_{k+2}) = 0,
\end{multline*}
where we use condition (\ref{contraction-eq}) for $k+1$ and antisymmetry of the last sum with respect to exchanging $i_{k+1}$ and $i_{k+2}$. Condition (\ref{contraction-eq}) holds for $F(i_0,\ldots, i_k)$ by its definition and the property of $H$. The antisymmetry of $F$ is preserved if we assume $H$ to have the property $H(-z) = -H(z)$, which is obvious from the construction.

Applying Lemmas~\ref{bounded-vol} and \ref{filling} we obtain for any fixed $i_0$:
\begin{multline*}
S(i_0, k) = \sum_{i_1,\ldots, i_k} \| F(i_0, i_1, \ldots, i_k) \| \le \\
\le (k+1) \left(
\sum_{i_1,\ldots, i_k} \|C(i_0, i_1, \ldots, i_k)\| + \sum_{i_1,\ldots, i_{k+1}} \|F(i_0, i_1, \ldots, i_k, i_{k+1})\|  \right) \le \\ 
\le (k+1)! g(d, k) \alpha n^{k-m} + (k+1) S(i_0, k+1) \le (k+1)! g(d, k) \alpha + (k+1) S(i_0, k+1),
\end{multline*}
because in the first sum every face of $C_{i_0}^{(d-k)}$ appears at most $k!$ times. If we use the shifted cubic partition and Lemma~\ref{filling-rect} in place of Lemma~\ref{filling} even better inequality is obtained: 
$$
S(i_0, k) \le (k+1)! g(d, k) \alpha n^{k-m} + S(i_0, k+1).
$$

Starting from $S(i_0, m+1) = 0$ we obtain by descending induction on $k$ that $S(i_0,k)$ is bounded from above by a value $\alpha h(d, k)$ not depending on $n$. Moreover, if we still want to investigate the dependence on $n$ then for $n\to \infty$ we have (using the shifted cubic partition):
\begin{equation}
\label{h-d1n}
h(d, 1, n) \le (m+1)! \binom{d}{m} 2^{2m-1} + O(1/n).
\end{equation}

To finish the proof we consider two equal (because of double counting of pairs $(i,j)$ and $(j, i)$ and the antisymmetry) sums:
$$
X = \sum_i \left( C_i - \sum_{j} F(i, j)\right) = \sum_i C_i = Q.
$$
In the first sum every summand $X_i = C_i + \sum_{j} F(i, j)$ has volume at most $\|C_i\| + h(d,1)\alpha \le \bar f(d)\alpha$ with some large value $\bar f(d)$. In the formula
$$
\partial X_i = \partial (C_i - \sum_{j} F(i, j)) = \partial C_i - \sum_{j} C_{i, j} + \sum_{j,k} F(i, j, k) 
$$
we use antisymmetry in the indices $j$ and $k$ in the last summand to obtain:
$$
\partial X_i = \partial C_i - \sum_{j} C_{i, j}  = 0.
$$

If $\bar f(d)\alpha < 1$ (thus $f(d, m)$ in the statement of the theorem may be chosen to be $1/\bar f(d)$) then $\|X_i\|<1$ and $X_i$ must be a $d$-dimensional boundary, because every $d$-cycle representing a nontrivial homology has volume at least $1$. Hence the sum of $X_i$'s cannot give $Q$, which is a nontrivial generator of $H_d(Q, \partial Q)$.

From (\ref{h-d1n}) we obtain the following explicit estimate with dependence on $n$:
\begin{equation}
\label{asymptotic-f}
f(d,m,n) = \left( (m+1)! \binom{d}{m} 2^{2m-1} + O(1/n) \right)^{-1}.
\end{equation}

\end{document}